\documentclass[10pt,oneside]{amsart}
\usepackage{jmgstandard}
\usepackage{multirow, verbatim,enumerate}
\usepackage{amscd, amssymb}
\usepackage[T1]{fontenc}

\title{Log canonical thresholds of del Pezzo Surfaces in characteristic p}
\date{29 December 2013}

\author{Jesus Martinez-Garcia}
\address{Johns Hopkins University\newline
\indent Department of Mathematics\newline
\indent 3400 N. Charles St. \newline
\indent Baltimore 21218\newline
\indent USA}%
\email{jmart168@math.jhu.edu}%
\urladdr{http://math.jhu.edu/jmartinezgarcia/}

\begin{document}

\begin{abstract}
The global log canonical threshold of each non-singular complex del Pezzo surface was computed by Cheltsov. The proof used Koll\'ar-Shokurov's connectedness principle and other results relying on vanishing theorems of Kodaira type, not known to be true in finite characteristic.

We compute the global log canonical threshold of non-singular del Pezzo surfaces over an algebraically closed field. We give algebraic proofs of results previously known only in characteristic $0$. Instead of using of the connectedness principle we introduce a new technique based on a classification of curves of low degree. As an application we conclude that non-singular del Pezzo surfaces in finite characteristic of degree lower or equal than $4$ are K-semistable.
\end{abstract}

\maketitle
%\pagenumbering{arabic}
\section{Introduction and definitions}
The log canonical threshold of an algebraic variety\footnote{Unless otherwise mentioned, all varieties in this article are projective normal varieties over an algebraically closed field $k$ of characteristic $p\geq 0$.} $X$, $\glct(X)$, is a numerical invariant introduced by Shokurov in the setting of the Minimal Model Program.

\begin{dfns}\label{dfn:logResol}
A \emph{log pair} $(X,D=\sum d_i D_i)$ is a pair where $X$ is a variety and $D\subset X$ an effective $\bbQ$-divisor which is $\bbQ$-Cartier.

A \emph{log resolution} of $(X, D)$ is a proper birational morphism $\sigma\colon Y \ra X$ such that $Y$ is non-singular, the support of the strict transform $\widetilde D:=(\sigma^{-1})_*(D)$ of $D$ is non-singular, the exceptional set $\Ex(\sigma)$ has pure codimension one, and $\Ex(\sigma)\cup\Supp(\widetilde D)$ intersects with simple normal crossings.
\end{dfns} 
An \emph{embedded resolution} of $(X,D)$ is a resolution in which $\Supp(D)$ is seen as a subvariety of $X$ and $X$ and $D$ are resolved \emph{at the same time}. Once an embedded resolution is found it can be easily modified into a log resolution.

Finding embedded resolutions or abstract resolutions of a pair $(X,D)$ over an algebraically closed field $k$ is an open problem. Embedded resolutions exist when $\charac(k)=0$ for all dimensions. For algebraically closed fields of finite characteristic they exist when $\dim(X)\leq 3$ thanks to a recent result by Cossart and Piltant \cite{CossartPiltant1},\cite{CossartPiltant2}. Previously Abhyankar had shown that embedded resolutions exist when $\dim(X) =3$ and $\charac(k)> 5$ \cite{AbhyankarBook}.

Assume $X$ is normal and $\bbQ$-factorial and let $\sigma\colon Y\ra X$ be a proper birational modification of the pair $(X,D)$. We may write 
\begin{equation}
K_{Y}+ D_Y + \sum a(F_j,X,D)F_j \equiv \sigma^*(K_X+D),
\label{eq:logres}
\end{equation}
where the $F_j$ are exceptional divisors and $D_Y$ is the strict transform of $D$ in $Y$.\footnote{In the context of the Minimal Model Programme, the terms $a(F_j,X,D) F_j$ usually appear at the other side of the inequality. Our notation is consistent with the literature on computations of global log canonical thresholds, as in \cite{CheltsovLCTdP}.}
We call $-a(F_j,X,D)$, the \emph{discrepancy of $F_j$ with respect to $(X,D)$} and we often write $-a_j$ if no confusion is likely. Denote by $\disc(X,D)=\inf \{-a(F_j,X,D)\}$, the \emph{discrepancy} of the pair $(X,D)$ where the infimum is taken over all proper birational modifications of $X$.

Note that if there is a proper birational morphism $\sigma$ such that $-a_j<-1$ for some exceptional divisor $F_j$, then we can blow-up repeatedly non-singular loci in $F_j$ and its strict transforms to obtain a new proper birational morphism $\sigma'\colon Y' \ra X$ factoring through $Y$ and such that one of the discrepancies $-a_k=-a(F_k,X,D)$ is arbitrarily small. In particular we obtain that $\disc(X,D)=-\infty$.

\begin{dfn}\label{dfn:mmpPairs}
Let $(X,D=\sum d_i D_i)$ be a log pair. We say $(X,D)$ is \emph{log canonical} if $\disc(X,D)\geq -1$ and $d_i\leq 1$ for all $i$.

We say $(X,D)$ \emph{is log canonical at $p\in X$} if $(U,D\vert_U)$ is log canonical where $U$ is a Zariski open neighbourhood of $p$.
\end{dfn}
By \cite[Cor. 3.13]{KolPairs} it is enough to compute the discrepancies under one log resolution $\sigma\colon Y \ra X$ of $(X,D)$ to decide if the pair is log canonical.

\begin{dfn}\label{dfn:lcts}
\mbox{}
\begin{itemize}
	\item[(i)] The \emph{log canonical threshold of the pair} $(X,D)$ is
$$\lct(X,D)=\max\{\lambda \setsep (X,\lambda D) \text{ is log canonical}\}.$$
	\item[(ii)] The \emph{local log canonical threshold} of the pair $(X,D)$ at $p\in X$ is
$$\lct_p(X,D)=\max\{\lambda \setsep (X,\lambda D) \text{ is log canonical at } p\}.$$
	\item[(iii)] The \emph{global log canonical threshold} of $X$ is
$$\glct(X)=\sup\left\{\lambda \ \left|
\aligned
&(X,\lambda D) \text{ is log canonical for all }\\
&\text{ effective }\bbQ\text{-divisors }D\simq -K_X \\
%&\text{the log pair}\ \left(X, \lambda D\right)\ \text{has log canonical singularities}\\
%&\text{for every effective $\mathbb{Q}$-divisor}\ D\sim_{\mathbb{Q}} H\\
\endaligned\right.\right\}.
$$
\end{itemize}
\end{dfn}

Obviously $\lct(X,D)\leq \lct_p(X,D)$ for all $p\in X$. Notice that if $(X,0)$ is not log canonical, then $\glct(X)=-\infty$. It is not yet known whether $\glct(X)$ is a rational number. The following conjecture, first stated by Tian for complex Fano manifolds, generalises the one in \cite[Conj. 1.4]{Cheltsov-Shramov-Park-WPS} for complex Fano varieties:
\begin{conj}
\label{conj:LCTrational}
Let $X$ be a projective Fano variety over an algebraically closed field. Suppose $X$ is $\bbQ$-factorial and has at worst log terminal singularities. Then 
\mbox{}
\begin{itemize}
	\item[(i)] there is an effective divisor $D\in \vert-mK_X\vert$ for some $m\in\bbN$ such that $\glct(X)=\lct(X,\frac{1}{m}D)$, and
	\item[(ii)] the global log canonical threshold $\glct(X)$ is a rational number.
\end{itemize}
\end{conj}
Note that for any effective $\bbQ$-divisor $D$, $\lct(X,D)$ is a rational number. Therefore part (i) of Conjecture \ref{conj:LCTrational} implies part (ii). This conjecture is not known to be true even for complex del Pezzo surfaces. However, there is strong evidence to support it. In fact, in the case of complex del Pezzo surfaces with Du Val singularities, $D$ can be found in $\vert -mK_X\vert$ for $m\leq 6$ (see \cite{Kosta-thesis} and \cite{Park-Won-lctDuVal}). In this article we verify Conjecture \ref{conj:LCTrational} for non-singular del Pezzo surfaces over an algebraically closed field. Our computation shows that in this case we may take $m=1$.

In \cite[Thm. A.3]{CheltsovShramovLct3folds}, Demailly gives an elegant proof of the following result: the global log canonical threshold of a smooth complex Fano variety $X$ coincides with Tian's $\alpha$-invariant $\alpha(X)$ introduced in \cite{TianAlphaInvariant}. Tian's $\alpha$-invariant is defined in a differential-geometric context. The main result in \cite{TianAlphaInvariant} is that if $\alpha(X)>\frac{n}{n+1}$ where $n$ is the dimension of $X$, then $X$ can be equipped with a K\"ahler-Einstein (KE) metric.

Ten years later, Tian proved \cite{TianKEimpliesAnalyticKstability} that the existence of a KE metric in a non-singular Fano variety is a sufficient condition for $X$ to be analytically $K$-stable. The definition of $K$-stability is rather technical and it involves the use of certain deformations of $X$ known as \textit{test configurations}, so we refer the reader to \cite{OdakaAnnals} for a detailed account.

Very recently, Chen, Donaldson and Sun (see \cite{Chen-Donaldson-Sun-Kstability}), and Tian in \cite{Tian-Kstability-solution} have independently proved that the existence of a K\"ahler-Einstein metric on a smooth complex projective variety $X$ is equivalent to $X$ being K-stable. A direct algebraic proof of the relation between the global log canonical threshold and K-stability avoiding K\"ahler-Einstein metrics is known:
\begin{thm}[\cite{OdakaSanoAlphaInvariant}]
\label{thm:OdakaSanoResult}
Let $X$ be a $\mathbb{Q}$-Fano variety of dimension $n$ and  suppose that $\glct(X)> \frac{n}{n+1}$ (resp. $\glct(X)\geq \frac{n}{n+1}$). 
Then, $(X,\mathcal{O}_{X}(-K_X))$ is K-stable $($resp.\ K-semistable$)$. 
\end{thm}
The proof uses resolution of singularities for dimension $n$, so it is valid in finite characteristic when $\dim(X)\leq 3$.

Although it is introduced above in the context of K\"ahler-Einstein metrics, $K$-stability is interesting in birational geometry on its own right. For instance, in \cite{OdakaAnnals}, Odaka shows that, given certain conditions, if $(X,L)$ is $K$-stable where $L$ is an ample line bundle then $X$ has only semi-log canonical singularities (the proof assumes $\charac(k)=0$).

The purpose of this article is to study the global log canonical thresholds of del Pezzo surfaces, proving the following result:
\begin{thm}[Main Theorem]
\label{thm:Main}
Let $S$ be a non-singular del Pezzo surface over an algebraically closed field $k$. Then:
\begin{equation*}
\glct(S)=\omega:=\left\{
	\begin{array}{rl}
		5/6 &\text{when } K_S^2=2 \text{ and } \vert -K_S\vert \text{ has no tacnodal curves,}\\
		3/4 &\text{when } K_S^2=2 \text{ and } \vert -K_S\vert \text{ has some tacnodal curve,}\\
		2/3 &\text{when } K_S^2=4,\\
\end{array}
\right.
\end{equation*}
\end{thm}
Cheltsov computed the global log canonical threshold of all complex non-singular del Pezzo surfaces \cite[Thm. 1.7]{CheltsovLCTdP}. When $K_S^2\neq 2,4$, Cheltsov uses Skoda's inequality (Lemma \ref{lem:adjunction}, for which we give an algebraic proof for surfaces), and lemmas \ref{lem:Convexity} and \ref{lem:logcanUpDown}, as well as the classification of del Pezzo surfaces (see Theorem \ref{thm:dPclassification}) as \emph{black boxes} in his proof. In this article we show that all these tools hold for algebraically closed fields. Therefore, together with Theorem \ref{thm:Main} we obtain:
\begin{cor}
\label{cor:main}
Let $S$ be a non-singular del Pezzo surface over an algebraically closed field $k$. Then:
\begin{equation*}
\glct(S)=\omega:=\left\{
	\begin{array}{rl}
		1		&\text{when } K_S^2=1 \text{ and } \vert -K_S\vert \text{ has no cuspidal curves,}\\
		5/6 &\text{when } K_S^2=1 \text{ and } \vert -K_S\vert \text{ has some cuspidal curve,}\\
		5/6 &\text{when } K_S^2=2 \text{ and } \vert -K_S\vert \text{ has no tacnodal curves,}\\
		3/4 &\text{when } K_S^2=2 \text{ and } \vert -K_S\vert \text{ has some tacnodal curve,}\\
		3/4 &\text{when } K_S^2=3 \text{ and } \forall C \in \vert -K_S\vert,\ C \text{ has no Eckardt points,}\\
		2/3 &\text{when } K_S^2=3 \text{ and } \exists C \in \vert -K_S\vert \text{ with some Eckardt point,}\\
		2/3 &\text{when } K_S^2=4,\\
		1/2 &\text{when } K_S^2=5,6 \text{ or } S\cong\bbP^1\times\bbP^1\ (K_S^2=8), \\
		1/3 &\text{when } K_S^2=7,9 \text{ or } S\cong\bbF_1\ (K_S^2=8).
\end{array}
\right.
\end{equation*}
\end{cor}

Given that Theorem \ref{thm:OdakaSanoResult} is valid in finite characteristic for varieties of dimension smaller or equal than $3$, we conclude the following:
\begin{cor}
\label{cor:k-stability}
Let $S$ be a non-singular del Pezzo surface over an algebraically closed field. If $K_S^2\leq 4$, then $S$ is K-semistable. If $K_S^2\leq 2$ or $K_S^2=3$ and $S$ has no Eckardt points, then $S$ is K-stable.
\end{cor}
To the best of the author's knowledge these are the first examples of K-stable Fano varieties over fields of finite characteristic. It is important to stress the significance of this application. Testing K-stability on a variety $X$ from the definition requires computing the Donaldson-Futaki invariant of all test configurations $\calX$. These are deformations of $X$ over $\bbA^1_k$, satisfying very mild properties. As a result, testing K-stability from the definition is impractical, since it is difficult to obtain a meaningful classification all test configurations for any given variety. When $k=\bbC$, the most common approach is to find a K\"ahler-Einstein metric. In finite characteristic this is not possible and, albeit limited, Theorem \ref{thm:OdakaSanoResult} is the only known general method. The classification of del Pezzo surfaces when $\charac(k)=p>0$ is the same as for $k=\bbC$ (see Theorem \ref{thm:dPclassification}). Therefore, following the classification of K-stable non-singular complex Fano surfaces, the author expects the following conjecture to be true:
\begin{conj}
\label{conj:K-stability-del-Pezzo}
Let $S$ be a non-singular del Pezzo surface over an algebraically closed field of finite characteristic. Then $S$ is K-stable if and only if $S$  is not the blow-up of $\bbP^2$ in one or two points.
\end{conj}

The first result for del Pezzo surfaces in the direction of Theorem \ref{thm:Main} appeared in \cite{ParkPezzoFibration} where Park showed that $(S,\omega D)$ is log canonical for $D\in \vert -K_S\vert$. Although throughout his article it was assumed that the ground field $k=\bbC$, the proof of this particular result did not depend on transcendental methods.

\subsection{Organisation of the article and techniques used}
In section \ref{sec:basic_tools} we remind the reader the basic classification of del Pezzo surfaces. Furthermore, we introduce general results in log canonicity that we will use throughout the article.  In section \ref{sec:degree-4}, the main part of the article, we deal with the case $K_S^2=4$. We finish the article with section \ref{sec:degree-2}, where we provide a simple proof for $K_S^2=2$ which is independent of the characteristic of the field.

Let us comment on the technique that Cheltsov uses to prove Corollary \ref{cor:main} when $K_S^2=4$ and $k=\bbC$ as well as the obstructions that make this unsuitable in finite characteristic. We also explain our basic approach to overcome those obstructions. Roughly speaking, from \cite{ParkPezzoFibration}, we can find an effective $D\in\vert-K_S\vert$ such that $(S,\omega D)$ is \textit{strictly log canonical}, i.e. $\lct(S,\omega D)=1$. To show that $\glct(S)=\omega=\frac{2}{3}$ Cheltsov proceeds by \textit{reductio ad absurdum}, supposing there is an effective $\bbQ$-divisor $D\simq-K_S$ such that $(S,\omega D)$ is not log canonical and deriving a contradiction. Let us introduce the following definition:
\begin{dfn}\label{dfn:LCS}
The \emph{non-klt locus} of a log pair $(X,D=\sum d_i D_i)$ as in \eqref{eq:logres} is the closed set:
$$
\nonKLT(X,D)=\left(\bigcup_{d_i\geq 1} D_i\right)\cup \left(\bigcup_{a_j\leq -1}\sigma(F_j)\right)\subsetneq X,
$$
where $\sigma$ is any proper birational modification. The non-klt locus is called the \textit{locus of log canonical singularities} in \cite{CheltsovLCTdP}.
\end{dfn}
Cheltsov uses $D$ to construct an effective $\bbQ$-divisor $D'\simq D$ such that $-(K_S+\omega D')$ is ample and $\nonKLT(S,D')$ is disconnected. Then, he obtains a contradiction using the \emph{Koll\'ar-Shokurov connectivity principle}:
\begin{lem}[see, for instance, {\cite[Thm. 7.4]{KolPairs}}]
\label{lem:lcsConnectedComplex}
Let $X,Z$ be normal complex varieties and $f:X\ra Z$ be a contraction and $D'$ an effective $\bbQ$-divisor on $X$ such that $K_X+D'$ is $\bbQ$-Cartier. Assume $-(K_X+D')$ is  $f$-nef and $f$-big. Then $\nonKLT(X,D)$ is connected in a neighbourhood of any fibre of $f$.
\end{lem}
The proof of this lemma uses Kawamata-Viehweg vanishing theorem for $k=\bbC$. In characteristic $p$ there are counter-examples of Kawamata-Viehweg vanishing when $\dim(X)\geq 3$.

The main bulk of our article is section \ref{sec:degree-4}, where we deal with the case $K_S^2=4$. Subsection \ref{sec:degree-4-low-degree-curves} is preparatory, studying curves of low degree and birational morphisms $S\ra \bbP^2$ with certain properties. We then construct certain $\bbQ$-divisors $G$ and $H$ with certain properties  in subsection \ref{sec:degree-4-Aux-Qdivisors}. The construction is rather technical, but necessary for the proof of Theorem \ref{thm:Main}. Therefore the reader may want to skip it in a first read. Subsection \ref{sec:delPezzo-4-proof} contains the main part of the proof when $K_S^2=4$ and it follows the following approach: for all effective $\bbQ$-divisors $D\simq -K_S$, we first show that $\nonKLT(S,D)$ has codimension $2$. If for such a $\bbQ$-divisor $D$, the pair $(S,D)$ is not log canonical at some point $p\in \Supp(D)$, we use intersection theory on $D$ and the $\bbQ$ divisors $G$ and $H$ constructed previously to obtain a contradiction, using the results in section \ref{sec:basic_tools}. A crucial point in the proof is that $G$ and $H$ contain $p$ in their support.

\section*{Acknowledgements}
I would like to thank Ivan Cheltsov for introducing me to this problem as well as for his invaluable support and advice. I would also like to thank Yuji Sano for clarifications regarding \cite{OdakaSanoAlphaInvariant} and an anonymous referee for useful and detailed comments.

This research is part of the author's Ph.D. Thesis \cite{JMGthesis}.
\newpage
\section{Basic Tools}
\label{sec:basic_tools}
\subsection{Results in log canonicity}
\label{sec:logcan-results}
\begin{nota}
\label{nota:Basics}
Let $S$ be a non-singular surface. Let $f:\widetilde S\ra S$ be a proper birational morphism and $D$ an effective $\bbQ$-divisor in $S$ with proper transform $\widetilde D$. Then we can write \eqref{eq:logres} as
$$K_{\widetilde S}+ \widetilde D +\sum^{r}_{i=1}a_iE_i\equiv f^*(K_S+D),$$
where $E_i$ are exceptional curves ($E_i\cong \bbP^1, E_i^2<0)$ and $a_i$ are rational numbers.

Often $f\colon \widetilde S \ra S$ will be the blow-up of a point $p$ with exceptional curve $E$. Other times $f$ will be the minimal log resolution of $(S,D)$. This will be clear from the context, when not explicitly stated. We will denote the strict transform of any $\bbQ$-divisor $B$ in $\widetilde S$ by $\widetilde B$.
\end{nota}

\begin{lem}\label{lem:logcanUpDown}
The log pair $(S,D)$ is log canonical if and only if 
		\begin{equation}
						(\widetilde S, \widetilde D + \displaystyle{\sum^{r}_{i=1}} a_iE_i)
				\label{eq:logcanUpDown}
		\end{equation}
is log canonical. In particular when $f\colon \widetilde S \ra S$ is the blow-up of a point $p\in S$ with exceptional divisor $E$, the pair $(S, D)$ is log canonical at $p$ if and only if
\begin{equation}
(\widetilde S, \widetilde D +(\mult_pD-1)E)
\label{eq:BlowupUpDownLC}
\end{equation}
is log canonical for all $q\in E$.
\end{lem}
%
%In this section we deal with a pair $(S,D)$ (or $(S,\omega D)$ for some $\omega \in \bbQ\cap [0,1]$) where $S$ is a non-singular surface. Let $p\in S$ and $D\simq -K_S$ be an effective $\bbQ$-divisor such that $(S,D)$ (respectively $(S,\omega D)$) is not log canonical at $p$.

It is well known that log canonical pairs satisfy a convex property:
\begin{lem}
\label{lem:convexity-forward}
Let $S$ be a non-singular surface, $D$ and $B$ be effective $\bbQ$-divisors on $S$. If $(S,D)$ and $(S,B)$ are log canonical then, for all $\alpha\in [0,1]\cap\bbQ$, the pair
$$(S,\alpha D+(1-\alpha)B)$$
is log canonical.
\end{lem}
\begin{proof}
Let $f\colon \widetilde S\ra S$ be any proper birational morphism with exceptional divisor $\bigcup E_i$. Then we may write
$$\alpha K_{\widetilde S} + \alpha \widetilde D + \alpha \sum a_i E_i\simq \alpha f^*(K_S+D) ,$$
$$(1-\alpha) K_{\widetilde S} + (1-\alpha) \widetilde B+(1-\alpha) \sum b_i E_i\simq (1-\alpha) f^*(K_S+B),$$
where $a_i\leq 1$, $b_i\leq 1$, since $(S,D)$, $(S,B)$ are log canonical.
Adding the two equivalences, we bound the discrepancies of $(S,\alpha D + (1-\alpha)B)$:
\[\alpha a_i + (1-\alpha) b_i \leq \alpha +(1-\alpha) = 1.\qedhere\]
\end{proof}
We will be interested in the contrapositive of this result:
\begin{lem}[Convexity]\label{lem:Convexity}
Given $S$ a non-singular surface (at a point $p$), let $D,B$ be effective $\bbQ$-divisors on $S$ such that $(S,B)$ is log canonical (at $p$) and $(S,D)$ is not log canonical (at $p$). Then, for all $\alpha \in [0,1)\bigcap \bbQ$ such that
$$D' = \frac{1}{1-\alpha}(D-\alpha B)$$
is effective, the pair $(S, D')$ is not log canonical (at $p$).
Moreover if $D\simq B$, then $ D'\simq D$ and we can choose $\alpha$ such that there is an irreducible curve $B_i$ in the support of $B$ (with $p\in B_i$), such that $B_i \not\subset \Supp( D')$.
\end{lem}
The following result is well known and it can be found (when the ground field is $\bbC$) on \cite{CheltsovLCTdP}. We provide an algebraic proof.
\begin{lem}\label{lem:adjunction}
Let $S$ be a non-singular surface, $D$ be an effective $\bbQ$-divisor and $C$ be an irreducible curve on the surface $S$. We may write $D = mC + \Omega$, where $m\geq 0$ is a rational number, and $\Omega=\sum a_i \Omega_i$ is an effective $\bbQ$-divisor such that $C \not\subset \Supp(\Omega)$. Suppose the pair $(S,D)$ is not log canonical at some point $p\in S$ such that $p\in C$. The following are true:
\begin{itemize}
	\item[(i)] $\mult_p D>1$.
	\item[(ii)] The curve $C \subset \nonKLT(S,D)$ if and only if $m \geq 1$. In particular, $(S,D)$ is not log canonical along $C$, if and only if $m>1$.
	\item[(iii)] If $m \leq 1$ and $p\in C$ with $C$ non-singular at $p$, then $C\cdot \Omega>1$.
\end{itemize}
\end{lem}
Inequality (i) is sometimes known in complex geometry as Skoda's inequality (see \cite{Demailly-Hiep}).
\begin{proof}
Part (ii) follows from the definition of $\nonKLT(S,D)$. For parts (i) and (iii) suppose $(S, D)$ is not log canonical at $p$. Consider $f\colon \widetilde S \ra S$, the minimal log resolution of $(S,D)$ around $p$, where the components of $f^{-1}(D)$ have simple normal crossings. By Lemma \ref{lem:logcanUpDown}, the pair
$$(\widetilde S, \widetilde D + \sum_{i=1}^N a_i E_i)$$
is not log canonical. We do induction on the number $N$ of exceptional divisors.

For the induction hypothesis we assume that (i) and (iii) hold if the minimal log resolution of a pair consists on at most $N$ blow-ups. Suppose the log resolution of $(S,D)$ consists of $(N+1)$ blow-ups. Let $\sigma \colon S_1 \ra S$ be the blow-up of $p$ with exceptional divisor $E_1$. Since the minimal log resolution is unique, $f$ factors through $S_1$, i.e. there is a birational morphism $g\colon \widetilde S \ra S_1$ consisting of $N$ blow-ups, such that $f=\sigma \circ g$. By Lemma \ref{lem:logcanUpDown}, the pair
\begin{equation}
(S_1, D_1 + (\mult_p D -1)E_1)
\label{eq:adjunction-proof}
\end{equation}
is not log canonical at some $q\in E_1$, where $D_1$, $C_1$ and $\Omega_1$ are the strict transforms of $D$, $C$ and $\Omega$, respectively.

We will prove (i) first, and then (iii). For (i), in the initial step of induction, $N=0$, $D$ is non-singular at $p$, so we can assume $D=aD_1$ around $p$. Since $(S,D)$ is not log canonical, $a>1$, so $\mult_p(D)=a>1$. For the inductive step, suppose that the log resolution of $(S,D)$ consists of $N+1$ blow-ups. Then the minimal log resolution of the pair \eqref{eq:adjunction-proof} consists of $N$ blow-ups. Therefore we may apply the induction hypothesis to show
$$1<\mult_qD_1 + (\mult_pD -1)\leq 2\mult_pD-1$$
which implies $\mult_pD>1$.

For part (iii) we observe that $\Supp(D)$ is singular at $p$. If it was not, then $D=mC$ near $p$ and $m>1$ since $(S,D)$ is not log canonical. Therefore the minimal log resolution of $(S,D)$ consists of $N\geq 1$ blow-ups. The initial step for the induction occurs when \eqref{eq:adjunction-proof} is already a log resolution ($N=1$). Then
$$1<\mult_pD-1=\mult_p\Omega+m-1\leq \mult_p\Omega \leq C\cdot \Omega,$$
proving the claim.

Suppose $\mult_pD-1\leq 1$. Then the pair \eqref{eq:adjunction-proof} is not log canonical at some point $q\in E_1$ and log canonical near $q$. The log resolution of the pair \eqref{eq:adjunction-proof} consists of $N-1\geq 1$ blow-ups and we can assume part (iii) in the statement is verified for \eqref{eq:adjunction-proof} by the induction hypothesis, where we substitute $C$ by $E_1$ or $C_1$, the strict transform of $C$ in $S_1$. Denote the strict transform of $\Omega$ in $S_1$ by $\Omega_1$. If $q\in C_1$, then by the induction hypothesis
$$1<C_1\cdot (\Omega_1+(\mult_pD-1)E_1)=C\cdot \Omega+m-1 \leq C\cdot \Omega$$
since $m\leq 1$. If $q\not\in C_1$, then the pair
$$(S_1, \Omega_1 + (\mult_pD -1)E_1)$$
is not log canonical at $q$ and by the induction hypothesis we have
$$1<E_1\cdot \Omega_1 = \mult_p\Omega\leq C\cdot \Omega.$$
\end{proof}

Since the first version of this manuscript, Lemma \ref{lem:adjunction} has been generalised, using the same approach of induction in the number of blow-ups:
\begin{thm}[{see \cite[Theorem 13]{Cheltsov-Trento-Proceedings} or \cite[Theorem 2.3.11]{JMGthesis}}]
\label{thm:inequality-Cheltsov}
Let $S$ be a surface and $p\in S$ be a non-singular point. Let
$$(S, a_1C_1+a_2C_2+\Omega)$$
be a log pair which is not log canonical at $p\in S$ but is log canonical near $p$. Suppose that $(C_1\cdot C_2)\vert_p=1$, $C_1, C_2$ are non-singular at $p$ and $C_1, C_2\not\subseteq\Supp(\Omega)$. Suppose that $a_1>0$, $a_2>0$ and $0<\mult_p\Omega\leq 1$. Then
$$(\Omega\cdot C_1)\vert_p>2(1-a_2)\qquad \text{ or }\qquad (\Omega\cdot C_2)\vert_p>2(1-a_1).$$
\end{thm}

%\begin{cor}\label{cor:MultPlusMult}
%If $(S,D)$ (respectively $(S,\omega D))$ is not log canonical at $p$ then, for $p$ and $q$ as in Notation \ref{nota:Basics}, we have
%$$\mult_q \widetilde D + \mult_p D >2\qquad (\text{respectively }\mult_q \widetilde D + \mult_p D >\frac{2}{\omega}).$$
%\end{cor}

\subsection{del Pezzo surfaces}
\label{sec:del-Pezzo-surfaces}
We recall some standard results of surfaces that we will use often.
\begin{dfn}
\label{dfn:DPdegree}
A \emph{del Pezzo surface} $S$ over an algebraically closed field $k$ is a non-singular surface whose anticanonical divisor, $-K_S$ is ample. Given any effective $\bbQ$-divisor $D\neq 0$, its \emph{anticanonical degree} (or just degree) is the positive rational number defined by
$$\deg (D)=(-K_S)\cdot D.$$
If $D$ is a divisor, $\deg(D)$ is a positive integer. The \emph{degree} of $S$ is the positive integer
$$\deg(S)=(K_S)^2.$$
\end{dfn}
We will call effective divisors of degrees $1,2,3,\ldots$ lines, conics, cubics... respectively.
\begin{thm}[{\cite[Chapter IV,Theorem 24.3 (ii)]{ManinCubicForms}}]
Let $S$ be a del Pezzo surface. Then every irreducible curve with a negative self-intersection number is exceptional.
\end{thm}
\begin{dfn}
A set of distinct points $\{p_1,\ldots ,p_r\}$ on $\bbP_k^2$ with $r\leq 8$ are in \emph{general position} if no three of them lie on a line, no six of them lie on a conic and a cubic containing $7$ points, one of them double, does not contain the eighth one.
\end{dfn}
We can classify del Pezzo surfaces:
\begin{thm}[{\cite[Chapter IV, Theorems 24.3, 24.4, 26.2]{ManinCubicForms}}]
\label{thm:dPclassification}
Let $S$ be a del Pezzo surface of degree $d$. Then $1\leq d\leq 9$ and either $S=\bbP^1\times \bbP^1$ ($\deg S=8$) or $S$ is a blow-up of $\bbP^2$ in $9-d$ points in general position $\pi\colon S\ra \bbP^2$.

Conversely, any blow-up of $\bbP^2$ in $9-d$ points in general position, for $1\leq d\leq 9$ is a del Pezzo surface of degree $d$. We call the morphism $\pi$ a \emph{model} of $S$.
\end{thm}
There are further ways of classifying del Pezzo surfaces. For instance, a del Pezzo surface $S$ has $\deg(S)=4$, if and only if $S$ is the non-singular complete intersection of two quadrics in $\bbP^4$. A del Pezzo surface has $\deg(S)=3$ if and only if $S$ is a non-singular cubic surface.

Theorem \ref{thm:dPclassification} implies that del Pezzo surfaces are rational. The following result applies:
\begin{prop}
\label{prop:curveDegDP}
For $S$ a non-singular rational surface and $C$ an effective divisor in $S$ with arithmetic genus $p_a(C)=0$, we have
\begin{equation}
h^0(S,\calO_S(C))\geq(-K_S)\cdot C.
\label{eq:sectionsC}
\end{equation}
\end{prop}
\begin{proof}
By Serre Duality:
$$h^2(S,\calO_S(C))=h^0(S,\calO_S(K_S-C))=0,$$
since $S$ is rational. By the Riemann-Roch theorem:
$$h^0(S,\calO_S(C))\geq \frac{1}{2}C\cdot(C-K_S)+1=-K_S\cdot C +p_a(C),$$
where we use the genus formula.
\end{proof}

\section{del Pezzo Surface of degree 4}
\label{sec:degree-4}
Let $S$ be a del Pezzo surface of degree $4$. In this section we prove $\glct(S)=\omega:=\frac{2}{3}$. We first classify low degree curves on $S$ (subsection \ref{sec:degree-4-low-degree-curves}). Then, we construct effective anticanonical $\bbQ$-divisors $G$ and $H$ with certain good properties (subsection \ref{sec:degree-4-Aux-Qdivisors}). These are used to prove Theorem \ref{thm:Main} when $K_S^2=4$ (subsection \ref{sec:delPezzo-4-proof}). 

\subsection{Curves of low degree and models of $S$}
\label{sec:degree-4-low-degree-curves}
Let $\pi\colon S\ra \bbP^2$ be the blow-up at points $p_1,\ldots,p_5\in \bbP^2$ in general position. Let $E_1,\ldots,E_5$ be the exceptional divisors. Recall $-K_S\sim \pioplane{3}-\sum^5_{i=1}E_i$ and $E_i^2=-1$.
\begin{table}[!ht]%
\begin{center}
\begin{tabular}{|c|c|c|c|c|c|}
\hline
Linear system $\calL\calS$																																			&$\deg C$	&$C^2$	&Fix $p$	&Fix $q$	&$C'$\\
\hline \hline
$\vert E_{i}\vert$																																										&$1$			&$-1$		&N					&N			&$E_i$\\
\hline
$\calL_{ij}= \left \vert \pioplane{1} -E_i-E_j\right\vert$																						&$1$			&$-1$		&N					&N			&$L_{ij}$\\
\hline
$\displaystyle{\calC_0= \left \vert \pioplane{2} -\sum^5_{i=1}E_i \right\vert}$											&$1$			&$-1$		&N					&N			&$C_0$\\
\hline
$\calB_{i}= \left \vert \pioplane{1} -E_i\right\vert$																								&$2$			&$0$		&Y					&N			&$B_i$\\
\hline
$\displaystyle{\calA_i = \left \vert \pioplane{2} -\sum^5_{\substack{j=1\\j\neq i}}E_j \right\vert}$	&$2$			&$0$		&Y					&N			&$A_i$\\
\hline
$\displaystyle{\calQ_i = \left \vert \pioplane{3} -E_i -\sum^5_{\substack{j=1}}E_j\right\vert}$			&$3$			&$1$		&Y					&Y			&$Q_i$\\
\hline
$\calR= \left \vert \pioplane{1}\right\vert$																													&$3$			&$1$		&Y					&Y			&$R$\\
\hline
$\calR_{ijk}= \left \vert \pioplane{2} -E_i-E_j-E_k \right\vert$																			&$3$			&$1$		&Y					&Y			&$R_{ijk}$\\
\hline
\end{tabular}
\end{center}
\caption{Catalogue of curves of low degree in $S$.}
\label{tab:delPezzo-4-lowdegree}
\end{table}

Observe Table \ref{tab:delPezzo-4-lowdegree}. In the first column we have defined certain complete linear systems $\calL\calS$ in $S$. Let $C\sim\calL\calS$ be any divisor. Its numerical properties ($C^2, \deg(C)$) are the same for any divisor in a given $\calL\calS$ and are easy to compute. We list them in the second and third columns of Table \ref{tab:delPezzo-4-lowdegree}. Note that, by the genus formula, $p_a(C)=0$ in all cases.

If $\deg C =1$, then by Proposition \ref{prop:curveDegDP}, $h^0(\calL\calS)\geq 1$. It is well-known that a del Pezzo surface has a finite number of lines (see \cite[Lemma 3.1.13]{JMGthesis} or \cite[Thm. V.4.9]{HartshorneAG} for cubic surfaces). Hence $h^0(\calL\calS)=1$ and we can find a unique curve $C'\in \calL \calS$. The notation for each particular $C'$ is in the last column of the table. 

If $\deg C=2$, then by Proposition \ref{prop:curveDegDP}, $h^0(\calL\calS)\geq 2$. Take $\calL\calS'\subset \calL\calS$ to be the sublinear system fixing $p$. Then $h^0(\calL\calS')\geq 1$ and we can find a curve $C'\in\calL\calS$ with $p\in C'$. The notation for each particular $C'$ is in the last column of the table. 

When the curve $C'$ is irreducible, we can realise it as the strict transform of an irreducible curve in $\bbP^2$ via the model $\pi$. For instance $L_{ij}$ is the strict transform of the unique line through $p_i$ and $p_j$. $C_0$ is the strict transform of the unique conic through all $p_i$. $B_i$ is the strict transform of a line passing through $p_i$ and $A_i$ is the strict transform of a conic through all $p_j$ but $p_i$. The last three rows of Table \ref{tab:delPezzo-4-lowdegree} deal with cubics and they are treated in Lemma \ref{lem:del-Pezzo-deg4-cubics}.

In order to understand the geometry of $S$ we need to understand which are its curves of low degree and how they intersect each other. We have just constructed some of these curves in Table \ref{tab:delPezzo-4-lowdegree}. In this section, among other properties of the low degree curves constructed above, we will show that the lines in Table \ref{tab:delPezzo-4-lowdegree} are all the lines in $S$. Furthermore, we will show that the conics in Table \ref{tab:delPezzo-4-lowdegree} are all the conics in $S$ passing through a given point $p$. Finally, there is more than one model $S\ra \bbP^2$ that characterises $S$ as a blow-up of the plane in $5$ points. We will also show how we can choose a model adequate to our needs.

\begin{lem}
\label{lem:del-Pezzo-deg4-good-curves-smooth}
Let $S$ be a del Pezzo surface of degree $4$ and $C'\in \calL\calS$ a curve as in Table \ref{tab:delPezzo-4-lowdegree}. Suppose $C'$ is irreducible. Then $C'$ is non-singular.
\end{lem}
\begin{proof}
Suppose $\calL\calS\neq \calQ_i$. Then, $\pi(C')$ is an irreducible curve of degree $1$ or $2$ in $\bbP^2$. Therefore $\pi(C')$ is non-singular. Since $S$ is just the blow-up of non-singular points of $\bbP^2$, if $\pi(C')$ is a non-singular curve of $\bbP^2$, then its strict transform $C'$ is a non-singular curve in $S$.

Suppose $C'=Q_i\in \calL\calS$. The irreducible curve $\pi(Q_i)$ is an irreducible cubic curve in $\bbP^2$ with multiplicity $2$ at $p_i$. Its strict transform $Q_i$ in $S$ must be non-singular, since it is enough to blow-up $S$ once at $p_i$ to resolve $\pi(Q_i)$.
\end{proof}

\begin{lem}[{see \cite[Lemma 3.3.2]{JMGthesis}}]
\label{lem:del-Pezzo-deg4-lines-list}
The $16$ lines in Table \ref{tab:delPezzo-4-lowdegree} are all the lines in $S$. The intersection of these lines are:
\begin{align*}
&E_i\cdot E_j = -\delta_{ij}, \qquad L_{ij}\cdot E_i = L_{ij}\cdot E_j = 1,\qquad C_0\cdot E_i=1,\qquad C_0^2=-1, \\
&C_0\cdot L_{ij}=0,\qquad {L_{ij}\cdot L_{kl}=\left\{ \begin{array}{rl}
																-1 &\text{if } i=k \text{ and } j=l, \\
																0 & \text{if only two subindices are equal,} \\
																1 & \text{if none of the subindices are equal.} \\
														\end{array} \right.}
\end{align*}
\end{lem}

\begin{lem}
\label{lem:del-Pezzo-deg4-model-line-choice}
Given a line $L\subset S$, we can choose a model $\gamma\colon S\ra \bbP^2$ such that $L=E_1$. If $p=L_1\cap L_2$, the intersection of two lines, we can choose $\gamma$ such that $L_1=E_1,\ L_2=L_{12}$.
\end{lem}
\begin{proof}
We construct $\gamma\colon S \ra \bbP^2$ by contracting $5$ disjoint exceptional curves $F_i$ (i.e. $F_i\cdot F_j=0$ if $i\neq j$). Let $F_1=L$. 
\begin{itemize}
	\item[(i)] If $F_1=E_1$, take $F_2=E_2, F_3 =L_{34}, F_4=L_{35}, F_5=L_{45}$.
	\item[(ii)] If $F_1=C_0$, take $F_j=L_{1j}$.
	\item[(iii)] If $F_1=L_{12}$, take $F_i=L_{1\ (i+1)}$ for $2\leq i\leq 4$, $F_5=C_0$.
\end{itemize}
Obvious relabelling exhausts all possibilities for $L$ among the $16$ lines in Lemma \ref{lem:del-Pezzo-deg4-lines-list}. By Castelnuovo contractibility criterion \cite[Thm. V.5.7]{HartshorneAG} we can contract each $F_i$, leaving every other point intact. The image of $\gamma$ is $\bbP^2$, because the relative minimal model of $S$ once $5$ exceptional curves are contracted is unique. For the second part we can assume already $L_1=E_1$ and run this lemma again. In that case we are in case (i) above and the lemma follows.
\end{proof}
As with Lemma \ref{lem:del-Pezzo-deg4-lines-list} one can show:
\begin{lem}[{see \cite[Lemma 3.3.4]{JMGthesis}}]
\label{lem:del-Pezzo-deg4-conics}
If $C$ is an irreducible conic in $S$ passing through $p$, then $C=A_i$ or $C=B_i$, with $\pi(C)$ either a conic through all marked points but $p_i$ or a line through $p$ and $p_i$, respectively.
\end{lem}

\begin{lem}
\label{lem:del-Pezzo-deg4-model-conic-choice}
Given $C$ an irreducible conic in $S$, $p\in C$, we can choose a model $\gamma\colon S\ra \bbP^2$ such that under that model the curve $C$ can be realised as $C=A_i$ for any $i$ in Table \ref{tab:delPezzo-4-lowdegree}, unless $p\in E_1$ in which case $i\neq 1$.
\end{lem}
\begin{proof}
If $p\in L$, a line in $S$, assume $L=E_1$ by Lemma \ref{lem:del-Pezzo-deg4-model-line-choice}. We have $C\neq A_1$ since otherwise  $$0=A_1\cdot E_1 =C \cdot E_1 \geq \mult_p(C)\cdot \mult_p(E_1)=1,$$
a contradiction.

If $C=B_1$, take $F_i$ and $\gamma:S\ra \bbP^2$ as in the proof of Lemma \ref{lem:del-Pezzo-deg4-model-line-choice}, case (i). Because $C$ is irreducible, $\overline C=\gamma(B_1)=\oplaned$ by the genus formula on $\bbP^2$. Moreover:
$$B_1\sim\gamma^*(\oplaned)-\sum_{i=1}^5 (F_i\cdot B_1) F_i= \gamma^*(\oplaned)-F_1-F_3-F_4-F_5,$$
and $2=B_1 \cdot (-K_S)=3d -4,$
so $d=2$.
Therefore under the new blow-up $C$ is $A_2$. By obvious relabelling of the $F_j$ we can consider $C=A_i$ with $i\neq 1$.

If $C=B_i$, with $i\neq 1$, then $p\not\in E_1$ since $C$ is irreducible. If $C=B_2$, the same choice of $F_i$ gives us $C=A_1$ under the new blow-up. If $C=B_i$ for $i=3,4,5$ take $F_1=E_1, F_2=E_i, F_3=L_{jk}, F_4=L_{jl}, F_5=L_{kl}$ for different $j,k,l\in \{1,\ldots,5\} \setminus \{ i\}$ and $C=A_i$ under $\gamma$.
\end{proof}

\begin{lem}
\label{lem:del-Pezzo-deg4-conics-tangency}
Let $p\in L$, where $L$ is a line, and let $C_1,C_2$ be distinct irreducible conics passing through $p$. Then $C_1$ and $C_2$ intersect normally at $p$.
\end{lem}
\begin{proof}
By Lemma \ref{lem:del-Pezzo-deg4-model-conic-choice} we may assume that $L=E_1$ and $C_1=A_i$ for some $2\leq i\leq 5$. Without loss of generality we may assume that $C_1=A_2$. Note that $C_2\neq A_1, B_j$, for $j>1$ since $E_1\cdot A_1=E_1\cdot B_j=0$ and $C_2$ being irreducible would give a contradiction:
$$0=E_1\cdot C_2\geq \mult_pE_1\cdot \mult_pC_2\geq1.$$
By Lemma \ref{lem:del-Pezzo-deg4-conics} we have that $C_2=B_1$ or $C_2=A_i$ for $i\neq 1,2$. In both cases $C_2\cdot A_2=1$, obtaining simple normal crossings at $p$:
$$1=C_2\cdot A_2\geq (C_2\cdot A_2)\vert_p.$$
\end{proof}

We finish this subsection with a study of the cubic curves of $S$.
\begin{lem}
\label{lem:del-Pezzo-deg4-cubics}
Let $\calL\calS$ be one of the complete linear systems of degree $3$ in the last three rows of Table \ref{tab:delPezzo-4-lowdegree}. Let $\sigma \colon \widetilde S \ra S$ be the blow-up of some point $p\in S$ with exceptional curve $E\subset \widetilde S$. Let $q\in E$. We distinguish the following cases:
\begin{itemize}
	\item[(1.1)] The point $p$ does not lie in any line and $q$ does not lie in the strict transform of any conic. Then there is $C'\in \calL\calS$ irreducible and non-singular with $p\in C'$ and $q\in \widetilde C'$.
	\item[(1.2)] The point $p$ does not lie on a line, $q\in \widetilde A_1$ and $q\not\in \widetilde B_1$. Then for $\calL\calS\in \{\calR_{1ij}, \calQ_1, \calR\}$ there is $C'\in \calL\calS$ irreducible and non-singular such that $p\in C'$ and $q\in \widetilde C'$.
	\item[(2.1)] The point $p\in E_1$ and no other line, $q\not\in \widetilde E_1$ and $q$ does not lie in the strict transform of any conic. Then for $\calL\calS\in \{\calR_{1jk}, Q_i\}$ there is $C'\in \calL\calS$ irreducible and non-singular such that $p\in C'$ and $q\in \widetilde C'$.
	\item[(2.2)] The point $p\in E_1$ and no other line, $q\not\in \widetilde E_1$ and $q\in \widetilde A_5$. Then for $\calL\calS\in \{\calR_{1j5}, \calQ_{5}\}$ there is $C'\in \calL\calS$ irreducible and non-singular such that $p\in C'$ and $q\in \widetilde C'$.
	\item[(2.3)] The point $p\in E_1$ and no other line, $q\in \widetilde E_1$. Then there is $Q_1\in \calQ_1$ irreducible and non-singular such that $p\in Q_1$ and $q\in \widetilde Q_1$.
\end{itemize}
\end{lem}
In each case, denote $C'$ by the letter in the last column of Table \ref{tab:delPezzo-4-lowdegree}. The cases considered can be expanded including, for instance, when $p$ is the intersection of two lines. However, for our purposes the current statement is sufficient. 

\begin{proof}
First observe that by Lemma \ref{lem:del-Pezzo-deg4-good-curves-smooth} all irreducible curves in $\calL\calS$ are non-singular. 

Let $\calL\calS'=\{D\in \calL\calS \setsep p\in \Supp(D)\}$ and let $\widetilde{\calL\calS'}=\vert\sigma^*(\calL\calS')-E\vert$. By Proposition \ref{prop:curveDegDP}
$$h^0(\widetilde{\calL\calS'})=h^0(\calL\calS')=h^0(\calL\calS)-1\geq 2,$$ so we can choose $B\in \widetilde{\calL\calS'}$ an effective divisor passing through $q$. We distinguish two cases.

\textbf{Case A: $E\not\subset\Supp(B)$.} Then let $C'=\sigma_*(B)$ and $B=\widetilde C' \sim \sigma^*(C')-E$ where $B\cdot E=1$. In particular $C'$ is non-singular at $p$. If $C'$ is irreducible, we are done. Suppose for contradiction that $C'$ is reducible, then $C'=L+F$, the union of a line $L$ and a (possibly reducible) conic $F$ not intersecting at $p$. Under the hypothesis of (1.1) this is impossible. Under the hypothesis of (1.2), $F=A_1$, but then $C'-A_1$ is not in the rational class of any line in Lemma \ref{lem:del-Pezzo-deg4-lines-list}, giving a contradiction.

Under the hypothesis of (2.1), if $p\in L$, then $L=E_1$, but since $q\in \widetilde C'$ and $q\not\in \widetilde E_1$, then $q\in \widetilde F$, which is impossible. Therefore $p\in F$ and hence $q\in \widetilde F$, which is also impossible, giving a contradiction.

Under the hypothesis of (2.2), we must have $F=A_5$, but then $\calL\calS-A_5$ should be the class of a line, which is impossible for $\calL\calS$ as in (2.2). 

Under the hypothesis of (2.3) either $q\in \widetilde E_1$ or $q\in \widetilde F$. Suppose the latter, we can assume, by using Lemma \ref{lem:del-Pezzo-deg4-conics} and Lemma \ref{lem:del-Pezzo-deg4-model-conic-choice} that $F= A_5\sim\pi^*(\oplane{2})-\sum_{i=1}^4E_i$.
$A_5$ is irreducible, since otherwise $L_a\sim A_5-E_1\sim \pioplane{2} -2E_1-\sum_{i=2}^4E_i,$ would be a line. Since $A_5\cdot E_1=1$, both curves intersect transversally at $p$. This is impossible, since $q\in \widetilde A_5\cap \widetilde E_1$. We conclude that $q$ does not belong to the strict transform of any conic. Therefore $L=E_1$, but then $F\sim\calQ_1-E_1=\pioplane{2}-3E_1-\sum_{i=2}^5 E_i$ and there is no conic in Lemma \ref{lem:del-Pezzo-deg4-conics} in that class, nor it is possible to find two lines in Lemma \ref{lem:del-Pezzo-deg4-lines-list} whose sum adds up to $F$. 

\textbf{Case B: $E\subset\Supp(B)$.} Let $B=A+bE$ where $E\not\subset\Supp(A)$, $b\geq 1$ is an integer and $A$ is effective. We want to show this is impossible under the different assumptions. Let $C'=\sigma_*(B)=\sigma_*(A)\in \calL\calS$. Then $\widetilde C'=A=B-bE\sim\sigma^*(C')-(b+1)E$, and $C'$ is singular at $p$ and reducible by Lemma \ref{lem:del-Pezzo-deg4-good-curves-smooth}. We write $C'=L+F$, the union of a line $L$ and a (possibly reducible) conic $F$ intersecting at $p$. Under the hypothesis of (1.1) and (1.2) this is impossible since $p$ does not belong to any line. Under the hypothesis of (2.1), (2.2) and (2.3), then $L=E_1$ and $F$ is irreducible, but in each case $F\sim\calL\calS-E_1$ is not the class of an irreducible conic by Lemma \ref{lem:del-Pezzo-deg4-conics}.
\end{proof}

\subsection{Auxiliary $\bbQ$-divisors}
\label{sec:degree-4-Aux-Qdivisors}
In this section we will use the rational curves constructed in the previous section to show the existence of certain effective anti-canonical $\bbQ$-divisors with good local properties and controlled singularities. These $\bbQ$-divisors are used in the proof of Theorem \ref{thm:Main}, when $K_S^2=4$.
\begin{lem}
\label{lem:del-Pezzo-deg4-complementary-hyperplane-section}
Given an integral curve $C\subset S$ with $\deg C \leq 2$, there is an irreducible curve $Z$ such that $Z+C\in \vert-K_S\vert$.
\end{lem}
\begin{proof}
Given $p\in C$, denote by $\sigma\colon\widetilde{S} \ra S$ the blow-up at $p$ and $\widetilde C$ the strict transform of $C$.

If $\deg C=1$ we can assume $C=E_1$ by Lemma \ref{lem:del-Pezzo-deg4-model-line-choice}. Consider
$$\calQ_1 = \vert\pioplane{3}-2E_1-E_2-\cdots -E_5\vert.$$
Choose $p\in E_1$ not passing through any other line and $q\in \widetilde E_1$. By Lemma \ref{lem:del-Pezzo-deg4-cubics} there is an irreducible and non-singular curve $Z=Q_1\in \calQ_1$.

If $\deg C=2$ by Lemma \ref{lem:del-Pezzo-deg4-model-conic-choice} assume $C=A_1$. Choose $p\in A_1$ such that $p$ is not in any line and take $Z=B_1\in \calB_1$, which is irreducible (see proof of Lemma \ref{lem:del-Pezzo-deg4-aux-divisors-G-Basic}, case 1). In both cases $C+Z\sim -K_S.$
\end{proof}

We provide a joint proof of the next two lemmas.
\begin{lem}
\label{lem:del-Pezzo-deg4-aux-divisors-G-Basic}
Let $S$ be a del Pezzo surface of degree $4$. Let $p\in S$ be a point belonging to at most one line. There is an effective $\bbQ$-divisor $G=\sum{g_i G_i}\simq-K_S$, where all $G_i$ are irreducible and non-singular curves and such that
\begin{itemize}
	\item[(i)] the pair $(S,\frac{2}{3}G)$ is log canonical,
	\item[(ii)] the point $p\in G_i$ for all $G_i$,
	\item[(iii)] $\deg G_i\leq 2$ for all $G_i$.
\end{itemize}
\end{lem}

\begin{lem}
\label{lem:del-Pezzo-deg4-aux-divisors-H}
Let $S$ be a del Pezzo surface of degree $4$ and $p\in S$ be a point which belongs to at most one line. Let $\sigma \colon \widetilde S \ra S$ be the blow-up of $p$ with exceptional curve $E$. Let $q\in E$ be a given point. There exists an effective $\bbQ$-divisor $H=\sum{h_i H_i}\simq -K_S$ where $H_i$ are irreducible and non-singular curves and such that:
\begin{itemize}
	\item[(i)] the pair $(S,\frac{2}{3}H)$ is log canonical,
	\item[(ii)] the point $p\in H_i$ for all $H_i$,
	\item[(iii)] $\deg H_i\leq 3$ for all $H_i$,
	\item[(iv)] if $\deg H_i>1$, then the point $q\in \widetilde H_i$, the strict transform of $H_i$ via $\sigma$.
\end{itemize}
\end{lem}

\begin{proof}[Proof of lemmas \ref{lem:del-Pezzo-deg4-aux-divisors-G-Basic} and \ref{lem:del-Pezzo-deg4-aux-divisors-H}]
We construct $G$ and $H$ by case analysis on the position of $p\in S$ and $q\in E$. We use curves from Table \ref{tab:delPezzo-4-lowdegree}, which are lines, conics and cubics. These were constructed depending on $p$ and $q$ and were possibly reducible. Conditions (ii) and (iii) will be clear by construction, as well as condition (iv), for $H$.

We will check that $(S, \frac{2}{3}G)$ and $(S, \frac{2}{3}H)$ are log canonical. Note that it is sometimes enough to check $\mult_p H\leq \frac{3}{2}$, since if $(S, \frac{2}{3}H)$ is not log canonical, then $\mult_pH>\frac{3}{2}$ by Lemma \ref{lem:adjunction} (i).

The main task is to show that the curves chosen for each particular case are indeed irreducible. Ultimately this is the reason for our break down into cases. Irreducible curves in Table \ref{tab:delPezzo-4-lowdegree} are non-singular by Lemma \ref{lem:del-Pezzo-deg4-good-curves-smooth}.

\begin{case}
\label{case:1}{\ \\}
\assump{$p$ is not in any line}In particular $p\not\in E_i$ for all $i$. Let $G=A_1+B_1\sim -K_S$. Since all curves $C\subset \Supp(G)$ are conics, if $C=L_a+L_b$, the sum of two lines, then $p$ is in one line, contradicting Assumption 1. Moreover $(S,\frac{2}{3}G)$ is log canonical, since $A_1$ and $B_1$ intersect either in a tacnodal point or with simple normal crossings. 

\begin{subcase}
\label{subcase:1.1}{\ \\}
\subassump{$q$ is not in the strict transform of conics in $S$ passing through $p$}Let $H=\frac{1}{2}R+ \frac{1}{6}\sum_{i=1}^5Q_i \simq -K_S.$
Again, \refAssump and \refSubassump assure that $R$ and $Q_i$ are irreducible. Finally
\[\mult_p(H) = \left[\frac{1}{2}+5\cdot \frac{1}{6}\right]=\frac{8}{6}<\frac{3}{2}.\]
\end{subcase}

\begin{subcase}\label{subcase:1.2}{\ \\}
\subassump{The point $q\in \widetilde C$, the strict transform of a conic in $S$}By Assumption 1, $q\not\in \widetilde L$, for $L$ a line in $S$. Without loss of generality assume $C=A_1$, which is irreducible (use Lemma \ref{lem:del-Pezzo-deg4-model-conic-choice}). Observe that given a conic $C'\neq A_1$ with $p\in C'$, then $A_1\cdot C'=1$ unless $C'=B_1$. If $(A_1\cdot B_1)\vert_p=1$ then $A_1$ is the only conic such that $q\in \widetilde A_1$. Suppose $C'\neq B_1$ or $C'=B_1$ with $(A_1\cdot B_1)\vert_p =1$. In particular, $q\not\in \widetilde C'$. Let
\[H=\frac{1}{2}A_1 + \frac{1}{2}R_{125}+ \frac{1}{2}R_{134}\simq -K_S.\]
All components of $\Supp(H)$ are irreducible by Lemma \ref{lem:del-Pezzo-deg4-cubics}. Finally
\[\mult_p(H) = \left[\frac{1}{2}+\frac{1}{2}+\frac{1}{2}\right]=\frac{3}{2}.\]

Suppose $(A_1\cdot B_1)\vert_p=2$. Then $q=\widetilde B_1\cap \widetilde A_1$. Let $H=A_1+B_1\sim -K_S$. 
Clearly $A_1,B_1$ are irreducible by Assumption 1 and $(S, \frac{2}{3} H)$ is log canonical by Case 1.
\end{subcase}
\end{case}
\begin{case}
\label{case:2}
Suppose $p\in L$, a line in $S$ and no other line. By Lemma \ref{lem:del-Pezzo-deg4-model-line-choice} we can consider $L=E_1$.
\forcenewline\assump{$p\in E_1$ and $p\not\in L$, any other line different than $E_1$}Take
\[G=\frac{1}{3}\sum^5_{j=2}A_j+ \frac{1}{3}B_1 + \frac{2}{3}E_1 \simq -K_S.\]
$B_1$ is irreducible, since otherwise
\[\pi^*(\oplane{1})-E_1\sim B_1=L_a+E_1,\]
and $L_a\sim\pioplane{1}-2E_1$ is a line in $S$ contradicting Lemma \ref{lem:del-Pezzo-deg4-lines-list}. 

The curves $A_j$ are irreducible too. If they were not irreducible, then
\[\pi^*(\oplane{2})-\sum_{\substack{k=1\\k\neq j}}^5 E_k\simq A_j=L_b+E_1,\]
where
$$p\not\in L_b=A_j-E_1 \sim \pi^*(\oplane{2})-2E_1-\sum_{\substack{k=2\\k\neq j}}^5 E_k$$
is a line, but there is no such a line in $S$, by Lemma \ref{lem:del-Pezzo-deg4-lines-list}. 

Since
$$E_1\cdot A_j=B_1\cdot E_1 = A_j\cdot B_1 =A_j\cdot A_k=1,\ j\neq 1, k\neq j, k\neq 1,$$
all curves in $\Supp(G)$ intersect each other transversely so we blow up once to obtain simple normal crossings:
$$\sigma^*(\lambda G+K_{\widetilde S})\simq\lambda \widetilde G + \left(\left(4\cdot \frac{1}{3}+ \frac{1}{3}+ \frac{2}{3}\right) \lambda-1\right) F_1= \lambda \widetilde G + (\frac{7}{10}\lambda-1)F_1,$$
and $\lambda=\frac{2}{3}$ gives $\disc(S,\lambda D)\geq -1$ and $(S, \frac{2}{3}G)$ is log canonical.

\begin{subcase}
\label{subcase:2.1}\forcenewline
\subassump{$q\not\in \widetilde C$ for $C$ any line or conic in $S$}In particular $q\not\in \widetilde E_1$. Let
$$H=\frac{1}{8}\sum_{2\leq j<k\leq 5}R_{1jk} + \frac{1}{8}\sum^5_{i=2}Q_i+ \frac{1}{4}E_1\simq -K_S$$
By Lemma \ref{lem:del-Pezzo-deg4-cubics}, all components of $\Supp(H)$ are irreducible. Moreover
$$\mult_p(H) = \left(\frac{1}{8}\cdot 6 + \frac{1}{8}\cdot 4+ \frac{1}{4}\cdot 1\right) = \frac{3}{2}.$$
\end{subcase}
\begin{subcase}
\label{subcase:2.2}\forcenewline
\subassump{$q\in \widetilde C$, for some conic $C$ in $S$ but $q\not \in \widetilde L$, for all lines $L$ in $S$}In particular $C$ is irreducible and $q\not \in \widetilde E_1$. By lemmas \ref{lem:del-Pezzo-deg4-conics} and \ref{lem:del-Pezzo-deg4-model-conic-choice} we can assume that 
$$C\sim\pi^*(\oplane{2})-\sum_{\substack{i=1\\i\neq k}}^5 E_i~\sim A_k, \text{ for } k\neq1.$$
where $p\in C=A_k,\ k\neq 1$, with $q\in \widetilde C$. Moreover, since $q\in \widetilde A_k$, \refSubassump assures $A_k$ is irreducible.

Without loss of generality, suppose $k=5$. Suppose there is another conic $C'$ in $S$ such that $p\in C'$, $q\in \widetilde C'$ and $C'\neq A_5$. Since $p\in C'\cap E_1$, by Lemma  \ref{lem:del-Pezzo-deg4-conics} either $C'=B_1$ or $C'=A_j$, for $j\neq 1,5$. However $A_5\cdot B_1=1$, $A_i\cdot A_5=1$ for $i\neq 1,5$. Therefore in both cases $C'$ and $A_5$ intersect transversely and $q\not\in \widetilde C'$. Let
\begin{equation*}
H=\frac{3}{5} A_5 + \frac{1}{5}\left(R_{125}+R_{135}+R_{145}\right)+ \frac{1}{5}Q_5 + \frac{2}{5}E_1\simq -K_S.
%\label{eq:2.2}
\end{equation*}
All components of $\Supp(H)$ are irreducible by Lemma \ref{lem:del-Pezzo-deg4-cubics}.

We show that $(S, \frac{2}{3}H)$ is log canonical. Let $\sigma_0:S_0\ra S$ be the blow up at $p$ with exceptional divisor $F_1$ with $q\in F_1$. Table \ref{tab:intersectNumbers} gives the intersection numbers in $S_0$. 
\begin{table}[htb]%
\begin{center}
\begin{tabular}{|r||c|c|c|c|c|c|c|}
\hline
										&$\widetilde {A_5}$	&$\widetilde {R_{125}}$	&$\widetilde {R_{135}}$	&$\widetilde {R_{145}}$	&$\widetilde {Q_5}$	&$\widetilde{E_1}$ &$F_1$	\\
\hline \hline
$\widetilde {A_5}$	& &1&1&1&1&0&1\\
\hline
$\widetilde {R_{125}}$	& & &1&1&1&0&1\\
\hline
$\widetilde {R_{135}}$	& & & &1&1&0&1\\
\hline
$\widetilde {R_{145}}$	& & & & &1&0&1\\
\hline
$\widetilde {Q_{5}}$	& & & & & &0&1\\
\hline
$\widetilde {E_{1}}$	& & & & & & &1\\
\hline
\end{tabular}
\caption{Intersection numbers for subcase 2.2.}
\label{tab:intersectNumbers}
\end{center}
\end{table}
Since all curves other that $E_1$ in Table \ref{tab:intersectNumbers} intersect normally and pass through $q$, we just need to blow up $q$ to obtain simple normal crossings. Let $\sigma:\widetilde S \ra S$ be the composition of both blow-ups and $F_2$ be the second exceptional divisor. Then:
\begin{align*}
\sigma^*(\lambda H+K_{\widetilde S})\simq\lambda \widetilde H &+ \left(\left(\frac{3}{5}+3\cdot \frac{1}{5}+ \frac{1}{5}+ \frac{2}{5}\right)\lambda -1\right) F_1+\\
&\left(\left(\frac{7}{5}+\frac{9}{5}\right)\lambda -2\right) F_2,
\end{align*}
and for $\lambda=\frac{2}{3}$, the pair $(S,\lambda H)$ is log canonical.
\end{subcase}

\begin{subcase}
\label{subcase:2.3}{ \ \\ }
Suppose that under \refAssump  $q\in \widetilde L$ for some line $L$ in $S$. Then $L=E_1$.
\forcenewline
\subassump{$q\in \widetilde E_1$}Suppose for contradiction that $q\in \widetilde C$ where $C$ is a conic in $S$. As in case \ref{subcase:2.1} we can assume, by using Lemma \ref{lem:del-Pezzo-deg4-conics} and Lemma \ref{lem:del-Pezzo-deg4-model-conic-choice} that $C= A_5\sim\pi^*(\oplane{2})-\sum_{i=1}^4E_i$.
$A_5$ is irreducible, since otherwise $L_a\sim A_5-E_1\sim \pioplane{2} -2E_1-\sum_{i=2}^4E_i,$ would be a line in Lemma \ref{lem:del-Pezzo-deg4-lines-list}. Since $A_5\cdot E_1=1$, $A_5$ and $E_1$ intersect transversally at $p$. This is impossible, since $q\in \widetilde A_5\cap \widetilde E_1$. We conclude that $q$ does not belong to the strict transform of any conic. Now, take
$$H=Q_1+E_1\sim -K_S.$$
By Lemma \ref{lem:del-Pezzo-deg4-cubics} the curve $Q_1$ is irreducible. The pair $(S,\frac{2}{3}H)$ is log canonical, since $Q$ and $E_1$ intersect each other at worst at a tacnodal point.
\end{subcase}
\end{case}
\end{proof}

\subsection{Proof of Theorem \ref{thm:Main} in degree $4$}
\label{sec:delPezzo-4-proof}
In this section we show that $\glct(S)=\frac{2}{3}$. We start by showing that $\glct(S)\leq \frac{2}{3}$. Take $p=E_1\cap L_{12}$ and the conic $A_2$, which is non-singular and irreducible. Consider $G=E_1+L_{12}+A_2\sim K_S$. The result follows, since $\glct(S)\leq \lct_p(S, G)=2/3$.

We need to show $\glct(S)\geq \frac{2}{3}$. We proceed by contradiction. Suppose there is an effective $\bbQ$-divisor
$$D=\sum{d_i D_i}\simq -K_S, \qquad d_i>0 \ \forall i$$
such that $(S,\lambda D)$ is not log canonical for some $\lambda<\frac{2}{3}$. Then $\nonKLT(S,\lambda D)\neq \emptyset$. 

\begin{lem}
\label{lem:deg4LCSpoints}
$\nonKLT (S,\lambda D)$ contains only isolated points.
\end{lem}
\begin{proof}
If $C\subset \nonKLT(S,\lambda D)$, where $C$ is a curve, then $C=D_i$ for some $D_i$ such that $d_i>\frac{3}{2}$ by Lemma \ref{lem:adjunction} (ii). Then
$$4=-K_S \cdot D = \sum d_i \deg(D_i)>\frac{3}{2}\deg(D_i),$$
so $\deg (D_i) \leq 2$.
Using Lemma \ref{lem:del-Pezzo-deg4-complementary-hyperplane-section} choose an irreducible curve $Z$ such that $D_i+Z$ is cut out by a hyperplane section of $S$ passing through $D_i$. We have $D_i+Z\sim -K_S\simq D$. In particular $\deg Z \geq 2$, so $Z\cdot D_j\geq 0$ for all irreducible $D_j$ (since only lines can have negative self-intersection). Then
$$(Z\cdot D)\geq d_i(Z\cdot D_i)=d_i(-K_S-D_i)\cdot D_i =d_i\left( \deg D_i -(\deg D_i-2) \right)=2d_i>3$$
by the genus formula, since all lines and conics in $S$ are rational. But  $Z\cdot D=4-\deg D_i\leq 3$, giving a contradiction.
\end{proof}

Let $p\in \nonKLT(S,\lambda D)$, i.e. the pair $(S, \lambda D)$ is not log canonical at some point $p$.
\begin{lem}
The point $p$ is not in the intersection of two lines.
\end{lem}
\begin{proof}
Suppose for contradiction that $p$ is the intersection of two lines. By Lemma \ref{lem:del-Pezzo-deg4-model-line-choice} we may choose $\pi\colon S \ra \bbP^2$ such that $p=E_1\cap L_{12}$. For $L=E_1,L_{12}$ we have that $L\subseteq\Supp(D)$ since otherwise
$$1=L\cdot D \geq \mult_pD>\frac{1}{\lambda}>1,$$
by Lemma \ref{lem:adjunction} (i). Hence we may write $D=aE_1+bL_{12}+\Omega$ where $a,b>0$ and $E_1,L_{12}\not\subseteq\Supp(\Omega)$.

Observe that the curve $A_2$ in Table \ref{tab:delPezzo-4-lowdegree} with $p\in A_2$ is irreducible, since otherwise there would be lines passing through $p$ with either of the following rational classes
$$A_2-E_1\sim\pioplane{2}-2E_1-E_3-E_4-E_5,$$
$$A_2-L_{12}\sim\pioplane{1}+E_2-E_3-E_4-E_5,$$
which is impossible by Lemma \ref{lem:del-Pezzo-deg4-lines-list}. Since
$$A_2\cdot E_1=A_2\cdot L_{12}=L_{12}\cdot E_1=1,$$
the pair $(S, \lambda(A_2+E_1+L_{12}))$ is log canonical for $\lambda\leq \frac{2}{3}$. Therefore, by Lemma \ref{lem:Convexity} we may assume that $A_2\not\subset\Supp(D)$. We conclude
\begin{equation}
2\geq D\cdot A_2\geq a+b+\mult_p\Omega\geq a+b.
\label{eq:del-Pezzo-glct-4-proof-pEckardt}
\end{equation}
Now observe that
$$1=E_1\cdot D \geq -a+b+\mult_p\Omega,$$
$$1=L_{12}\cdot D \geq a-b+\mult_p\Omega,$$
and adding these two equations it follows that $\mult_p\Omega\leq 1$. The hypotheses of Theorem \ref{thm:inequality-Cheltsov} are satisfied. Therefore one of the following holds:
$$2(1-\lambda a)<L_{12}\cdot (\lambda\Omega)=\lambda(1-a+b)$$
$$2(1-\lambda b)<E_1\cdot (\lambda\Omega)=\lambda(1+a-b).$$
Since the roles of $a$ and $b$ are symmetric, it is enough to disprove the latter equation to obtain a contradiction. Indeed, the last inequality implies
$$2<\lambda(1+a+b)\leq3\lambda<2$$
by \eqref{eq:del-Pezzo-glct-4-proof-pEckardt}, a contradiction.
\end{proof}

Let $G$ be the effective $\bbQ$-divisor in Lemma \ref{lem:del-Pezzo-deg4-aux-divisors-G-Basic}. Recall that $(S,\lambda G)$ is log canonical and that all irreducible components $G_j\subset \Supp(G)$ satisfy $p\in G_j$.
By Lemma \ref{lem:Convexity}, we can assume there is an irreducible curve $G_j\subset \Supp (G)$ such that $G_j\not\subset \Supp (D)$. Then
$$2\geq \deg G_j= (-K_S)\cdot G_j= D\cdot G_j\geq \mult_p (D) \cdot \mult_p(G_j)\geq \mult_p(D).$$
Therefore, we have bounded the multiplicity of $D$ at $p$:
\begin{equation}
2\geq \mult_p(D) \geq \frac{3}{2}.
\label{eq:boundD}
\end{equation}

Let $\sigma : \widetilde S \lra S$ be the blow-up of $p$ with exceptional divisor $E$. By Lemma \ref{lem:logcanUpDown} the pair
$$(\widetilde S,\lambda \widetilde D +(\lambda\mult_p D -1)E)$$
is not log canonical at some $q\in E$. By \eqref{eq:boundD}, the pair is log canonical near $q\in E$. Applying Lemma \ref{lem:adjunction} (i) to this pair we obtain:
\begin{equation}
\mult_q(\widetilde D)+\mult_pD>3.
\label{eq:boundDtransform}
\end{equation}

Given $p\in S$ and $q\in E\subset \widetilde S$ as above, we apply Lemma \ref{lem:del-Pezzo-deg4-aux-divisors-H} to obtain an effective $\bbQ$-divisor $H=\sum h_i H_i$ on $S$ such that $\deg(H_i)\leq 3$ for all $i$, $(S,\lambda H)$ is log canonical with $p\in H_j$ for all irreducible components $H_j$ and $q\in \widetilde H_j$ whenever $\deg H_j>1$. Observe that if $\deg H_i=1$, then $H_i\subset \Supp (D)$, since otherwise
$$\frac{3}{2}\leq\mult_p(D)\leq D\cdot H_i=1.$$
which is impossible.

Since $(S, \lambda H)$ is log canonical, by Lemma \ref{lem:Convexity} we may assume that there is $H_j\not\subset \Supp (D)$ such that $q\in \widetilde H_j$, $p\in H_j$ and $2\leq\deg(H_j)\leq 3$. Then
$$\widetilde H_j \cdot \widetilde D = H_j \cdot D - \mult_p(H_j)\cdot\mult_p (D) \leq 3-\mult_p(D).$$
But $\widetilde H_j\not\subset \Supp(\widetilde D)$, so 
$$3-\mult_p(D)\geq \widetilde H_j \cdot \widetilde D\geq \mult_	q (\widetilde D),$$
contradicting \eqref{eq:boundDtransform}. This completes the proof of Theorem \ref{thm:Main} when $K_S^2=4$.
\section{del Pezzo surface of degree 2}
\label{sec:degree-2}
Let $S$ be a del Pezzo surface of degree $2$. We prove Theorem \ref{thm:Main} for this case.
Let $\omega=\frac{3}{4}$ if $\vert -K_S\vert$ has some tacnodal curve and $\omega=\frac{5}{6}$ otherwise. By \cite{ParkPezzoFibration}, if there is a tacnodal curve $C \in \vert\-K_S\vert$, then $\lct(S,C)=\frac{3}{4}$. Otherwise we can take $C\in \vert\-K_S\vert$ a cuspidal rational curve and $\lct(S,C)=\frac{5}{6}$. Therefore $\glct(S)\leq\omega$. We need to show:
$$\glct(S)\geq\omega.$$
We proceed by contradiction. Suppose there is an effective $\bbQ$-divisor $D\simq -K_S$ such that $(S,\omega D)$ is not log canonical. Reasoning as in Lemma \ref{lem:deg4LCSpoints} we can show that the non-klt locus $\nonKLT(S,\omega D)$ consists of isolated points. Let $p\in S$ be one of these points. Let $\calC\subset \vert-K_S\vert$ be the sublinear system fixing $p$. Suppose there is a curve $C \in \calC$ singular at $p$. The curve $C$ is the union of two lines intersecting at a tacnode or with simple normal crossings, a cuspidal rational curve or a nodal curve, so $(S,\omega C)$ is log canonical. By Lemma \ref{lem:Convexity} we may assume there is one component of $C$ not in $\Supp(D)$. The curve $C$ is reducible, since otherwise:
$$2=C\cdot D \geq \mult_pD \cdot \mult_pC>\frac{2}{\omega}>2,$$
by Lemma \ref{lem:adjunction} (i). 
If $C=L_1+L_2$, the union of two lines intersecting at $p$, then $(S, \omega (L_1+L_2))$ is log canonical and by Lemma \ref{lem:Convexity} we may assume that $L_1\not\subset \Supp(D)$. Then
\begin{equation}
1=L_1\cdot D \geq \mult_pD>\frac{1}{\omega}>1,
\label{eq:deg2AuxA}
\end{equation}
giving a contradiction by means of Lemma \ref{lem:adjunction} (i).

Therefore, we may assume that all $C\in \vert-K_S\vert$ passing through $p$ are non-singular at $p$. Let $\sigma:\widetilde S \ra S$ be the blow-up at $p$ with exceptional divisor $E$. Let $\widetilde D$ be the strict transform of $D$ in $\widetilde S$. Lemma \ref{lem:logcanUpDown} implies that the pair
\begin{equation}
(\widetilde S, \omega \widetilde D + (\omega\,\mult_pD-1)E)
\label{eq:deg2-log-pullback}
\end{equation}
is not log canonical at some point $q\in E$. Choosing a general $C\in \vert-K_S\vert$ with $p$ in $C$ we obtain that $2=C\cdot D\geq\mult_pD$, so $\omega\mult_pD-1\leq 1$ and the pair \eqref{eq:deg2-log-pullback} is log canonical near $q$. Applying Lemma \ref{lem:adjunction} (i) to this pair, we conclude
\begin{equation}
\mult_q\widetilde D +\mult_pD > \frac{2}{\omega}.
\label{eq:deg2-log-pullback-adjunction}
\end{equation}
By Proposition \ref{prop:curveDegDP} pick $C\in \calC$ such that $q\in \widetilde C$. By Lemma \ref{lem:Convexity}, if $C$ is irreducible, then $C\not \subset \Supp(D)$ and by \eqref{eq:deg2-log-pullback-adjunction} we obtain 
$$2-\mult_pD=\widetilde C \cdot \widetilde D \geq \mult_q(\widetilde D ) >\frac{2}{\omega}-\mult_p(D),$$
a contradiction. Hence $C=L_1+L_2$, the union of two lines $p\in L_1$, $p\not\in L_2$. The intersection numbers are:
$$L_1\cdot L_2=2, \qquad L_1^2=L_2^2=-1.$$
Since $(S,\omega C)$ is log canonical, by Lemma \ref{lem:Convexity} we can assume $L_1\subset \Supp(D)$ since otherwise the computation in \eqref{eq:deg2AuxA} gives a contradiction. Then, by Lemma \ref{lem:Convexity}, we may assume that $L_2\not\subset\Supp(D)$. Therefore we may write $D=mL_1+\Omega$ with $m>0$, $L_1,L_2\not\subseteq \Supp(\Omega)$. Then
$$1=L_2\cdot D=2m+L_2\cdot \Omega \geq 2m,$$
so $m\leq \frac{1}{2}$. Applying Lemma \ref{lem:adjunction} (iii) to the pair \eqref{eq:deg2-log-pullback} with $\widetilde L_1$ we obtain a contradiction, finishing the proof:
\begin{align*}
1&<\widetilde L_1 \cdot ( \omega \widetilde \Omega + (\omega\mult_p D-1)E)\\
&=\omega(L_1\cdot \Omega -\mult_p\Omega+\mult_p\Omega+m)-1\\
&<(L_1\cdot D -mL_1^2+m)-1 =2m\leq 1.
\end{align*}

\bibliographystyle{hep}
\bibliography{Bibliography}

\end{document}